\documentclass[12pt,english]{article}
\usepackage{amsmath, amsthm, latexsym, amssymb,url}
\usepackage{amsfonts}
\usepackage{xypic}
\usepackage{epsfig}
\usepackage{titlesec}
\titleformat{\section}{\normalfont\Large\bfseries}{\S\thesection}{1em}{}

\usepackage{graphicx}
\usepackage{yhmath}
\usepackage{mathdots}

\input xy
\xyoption{all}

\newcommand{\shrinkmargins}[1]{
  \addtolength{\textheight}{#1\topmargin}
  \addtolength{\textheight}{#1\topmargin}
  \addtolength{\textwidth}{#1\oddsidemargin}
  \addtolength{\textwidth}{#1\evensidemargin}
  \addtolength{\topmargin}{-#1\topmargin}
  \addtolength{\oddsidemargin}{-#1\oddsidemargin}
 \addtolength{\evensidemargin}{-#1\evensidemargin}
  }

\shrinkmargins{.3}

\theoremstyle{plain}

\newtheorem{theorem}{Theorem}[section]
\newtheorem{corollary}[theorem]{Corollary}
\newtheorem{lemma}[theorem]{Lemma}
\newtheorem{proposition}[theorem]{Proposition}
\newtheorem{question}[theorem]{Question}

\newtheorem*{teo}{Theorem}

\newtheorem{definition}[theorem]{Definition}

\theoremstyle{remark}
\newtheorem{remark}[theorem]{Remark}

\theoremstyle{definition}
\newtheorem{example}[theorem]{Example}

\theoremstyle{fact}

\theoremstyle{claim}

\def \Z { \mathbb{Z}}
\def \Q { \mathbb{Q}}

\def \ker { \text{Ker}}

\def\Ses#1#2#3#4#5{{#1}\overset{{#2}}{\rightarrow} {#3}\overset{{#4}}{\rightarrow} {#5}}

\begin{document}

\thispagestyle{empty}
\setcounter{tocdepth}{7}

\title{The spinor genus of the integral trace.}
\author{Guillermo Mantilla-Soler}


\date{}

\maketitle

\begin{abstract}
Let $K$ be a number field of degree at least $3$. In this article we show that the genus of the integral trace form of $K$ contains only one spinor genus. Additionally we show that exactly $43\%$ (resp. $29\%$, resp. $58\%$) of quadratic (resp. real quadratic, resp. imaginary quadratic) fields have the same property.\\
\end{abstract}

 \section*{Introduction}

Let $K$ be a number field. The rational quadratic form $x \mapsto \mathrm{tr}_{K/\Q}(x^{2})$ has been extensively studied by 
several authors, see for example \cite{ba}, \cite{ba1}, \cite{conner}, \cite{epk}, \cite{galla}, \cite{Mau}, 
\cite{minreis} and \cite{Serre}. For arithmetic purposes, a finer invariant of a number field is its integral trace form i.e., the 
integral quadratic form obtained by restricting $\mathrm{tr}_{K/\Q}(x^{2})$ to the maximal order in $K$. Recent applications of 
the integral trace form can be found in the work of Bhargava and Shnidman (see \cite{bhashn}) where they count cubic orders using 
the {\it shape}, an invariant closely related to the integral trace form. Some other applications of the integral trace on cubic 
fields can be found in \cite{Manti} and \cite{Manti2}. Given a non-degenerate integral quadratic form $q$, it is of great interest to study the number of spinor classes on its 
genus. This is explained by the famous result of Eichler \cite{eichler} which says that for an indefinite form $q$ of dimension of at least $3$ the spinor genus and the isometry class coincide. In this article we analyze the spinor genus of the integral trace form of a number field.  Our main theorem is the following:


\begin{teo}[cf.Theorem \ref{principal}] 
Let $K$ be a non-quadratic number field. Then, the genus of the integral trace form of $K$ contains only one proper spinor genus.
\end{teo}

In the case of quadratic fields it is not necessary true that the genus and the proper spinor genus of the integral trace coincide, however they agree and differ for infinitely many quadratic fields. In fact, they coincide (resp. differ) for a positive proportion of quadratic fields. 

\begin{teo}[cf.Theorem \ref{teoremacaso2}] 
For $43\%$ (resp. $29\%$, resp. $58\%$) of quadratic (resp. real quadratic, resp. imaginary quadratic) fields the genus and proper spinor genus of the integral trace coincide.
\end{teo}

\section{Definitions, notations and basic facts}\label{prelim}

\subsection{Notation}

We summarize here the most important notation used in the paper.

\begin{itemize}

\item For $a \in \Q_{p}$ we have that $v_{p}(a)$ is the usual $p$-adic valuation. We will use $\mathbb{A}$ to denote the adele ring over $\Q$, and for a field $K$ we will denote by $G_{K}$ the absolute Galois group ${\rm Gal}(K^{sep}/K)$.  

\item Let $a_{1},..., a_{n}$ be elements of a ring $R$, which in practice will be a maximal order on a number field or a local field, the $R$-isometry class of a quadratic form $a_{1}x_{1}^{2}+...+a_{1}x_{1}^{2}$ will be denoted by $\langle a_1,...,a_n\rangle$. Whenever there is a possible ambiguity in the ring of definition of an isometry between two quadratic forms we will write $\cong_{R}$ to make it clear that the forms are considered to be over $R.$

\item The isometry class of the binary integral quadratic form $2xy$ over $\Z_{2}$,  the {\it hyperbolic plane},  will be denoted by $\mathbb{H}$.

\item Most of our notation for quadratic forms is adopted from \cite{cassels} or \cite{Om}. Unexplained terminology is either standard or can be found in either of the aforementioned references.  
 
\end{itemize}
 
\subsection{Equivalences between integral quadratic forms} We briefly recall the standard notions of equivalence between integral quadratic forms. 
As it is customary we use geometric language. Let $V$ be a non-degenerate quadratic space over $\Q$ with orthogonal (resp. special orthogonal)
group ${\rm O}(V)$ (resp. ${\rm SO}(V)$) and let $\mathcal{L}_{V}$ be the 
set of lattices of maximal rank inside $V$. Let ${\rm O}_{\mathbb{A}}(V)$ (resp. ${\rm SO}_{\mathbb{A}}(V)$) be the adelic orthogonal (resp. special orthogonal) 
group. The group ${\rm O}_{\mathbb{A}}(V)$ acts on $\mathcal{L}_{V}$ and the different notions of equivalence between 
lattices are just given by orbits in $\mathcal{L}_{V}$, under the restricted action, of certain subgroups of ${\rm O}_{\mathbb{A}}(V)$.  By using the usual diagonal embedding we may 
view ${\rm O}(V)$ as a subgroup of ${\rm O}_{\mathbb{A}}(V)$. The notions of class, proper class and genus can all be given in terms of the groups 
${\rm O} (V)$, ${\rm SO}(V)$ and ${\rm SO}_{\mathbb{A}}(V)$, however for the spinor genus we need the group $\Theta_{\mathbb{A}}(V)$ which we define next. 
Let ${\rm {\bf Spin}}(V)$ be the simply connected linear algebraic group over $\Q$ that is the universal cover of the algebraic group 
${\rm \bf {SO}}(V)$. Explicitly, there is an exact sequence of algebraic groups
\[1 \rightarrow \Ses{{\bf \mu_{2}}}{ }{{\rm \bf {Spin}}(V)}{ }{{\rm {\bf SO}}(V)} \rightarrow 1.\] By taking $G_{\Q}$-invariance in the geometric sequence induced by the above exact sequence 
(see \cite[Proposition 22.15]{invol}), we get a connecting homomorphism of Galois cohomology
\[{\rm SO}(V) \rightarrow  \rm{H}^{1}(\Q,{\bf \mu_{2}}) \cong \Q^{*}/ (\Q^{*})^{2}.\] The obtained homomorphism 
$\theta: {\rm SO}(V) \to \Q^{*}/ (\Q^{*})^{2}$ is the so called  {\it spinor norm.} By considering, at every prime $p$, 
the quadratic space $V_{p}:=V\otimes_{\Q}\Q_{p}$ we have, as above, the local spinor norm 
\[\theta_{p}: {\rm SO}(V_{p}) \to \Q_{p}^{*}/ (\Q_{p}^{*})^{2}.\]  Putting all the local spinor norms together we have a group 
homomorphism\footnote{The image of $\Theta_{\mathbb{A}}$ is in fact contained in 
$\mathbb{A^{*}}/\mathbb{A^{*}}^{2}.$} \[ \theta_{\mathbb{A}} : {\rm SO}_{\mathbb{A}}(V) \to \prod_{p} \Q_{p}^{*}/ (\Q_{p}^{*})^{2}\] with 
kernel denoted by $\Theta_{\mathbb{A}}(V):= \ker(\theta_{\mathbb{A}})$.

\begin{definition}

Let $\Lambda$ be a lattice in $\mathcal{L}_{V}$.  For a subgroup $G$ of ${\rm O}_{\mathbb{A}}(V)$ the orbit of $\Lambda$ under the action of 
$G$ is denoted by $G(\Lambda)$. 
\begin{enumerate}
\item The class of $\Lambda$ (resp proper class) is ${\rm cl}(\Lambda):= {\rm O}(V)(\Lambda)$ (resp ${\rm cl}^{+}(\Lambda):= {\rm SO}(V)(\Lambda)$).
\item The genus of $\Lambda$ (resp proper genus) is ${\rm gen}(\Lambda):= {\rm O}_{\mathbb{A}}(V)(\Lambda)$ (resp ${\rm gen}^{+}(\Lambda):= {\rm SO}_{\mathbb{A}}(V)(\Lambda)$). 
\item The spinor genus of $\Lambda$ (resp proper spinor genus) is ${\rm spin}(\Lambda):= {\rm O}(V)\Theta_{\mathbb{A}}(V)(\Lambda)$ 
(resp ${\rm spin}^{+}(\Lambda):= {\rm SO}(V)\Theta_{\mathbb{A}}(V)(\Lambda)$).
\end{enumerate} Given $M \in \mathcal{L}_{V}$ we say that $\Lambda$ and $M$ are in the same class (resp. proper class, genus, proper genus, spinor genus, proper spinor genus) if they belong to the same orbit by the respective subgroup.
\end{definition}

It is not difficult to see that ${\rm gen}(\Lambda)={\rm gen}^{+}(\Lambda)$ for any lattice $\Lambda$ (see \cite[Chapter 8 Corollary to Lemma 3.2]{cassels}) so the following are immediate consequences of the definitions:

\[\begin{matrix} {\rm cl}(\Lambda) & \subseteq & {\rm spin}(\Lambda) & \subseteq & {\rm gen}(\Lambda) \\ \rotatebox[origin=c]{270}{ $\supseteq$ }  & & \rotatebox[origin=c]{270}{ $\supseteq$ } & & \rotatebox[origin=c]{270}{ $=$ }\\  {\rm cl}^{+}(\Lambda) & \subseteq & {\rm spin}^{+}(\Lambda) & \subseteq & {\rm gen}^{+}(\Lambda). \end{matrix}\]

The relevance of spinor equivalence, see \cite{eichler}, comes from the fact that for $\Lambda$ indefinite and of dimension bigger than $2$ we have that \[{\rm cl}^{+}(\Lambda)= {\rm spin}^{+}(\Lambda) \quad {\rm and} \quad {\rm cl}(\Lambda)= {\rm spin}(\Lambda).\]

The following criterion is a very useful tool to determine when an integral lattice $\Lambda$ has ${\rm spin}^{+}(\Lambda) = {\rm spin}(\Lambda) = {\rm gen}(\Lambda).$

\begin{theorem}\label{mainhelp}

Let $V$ be a non-degenerate quadratic space over $\Q$, and let $\Lambda$ be an integral lattice in $V$. For a prime $p$ let $\Lambda_{p}:=\Lambda \otimes \Z_{p}$, and let $\theta_{p}$ be the spinor norm on the $\Q_{p}$-quadratic space $V_{p}:=V \otimes_{\Q} \Q_{p}$. Suppose that for all primes $p$ we have that\footnote{See \cite[Chapter 11 \S1 (1.20) and (1.21)]{cassels} for the definition of $\theta_{p}(\Lambda_{p})$.} $\theta_{p}(\Lambda_{p}) \supseteq \Z_{p}^{*}.$ Then, the genus of $\Lambda$ contains only one proper\footnote{What Cassels calls spinor genus, \cite[Chapter 11 \S1]{cassels}, is what is usually called proper spinor genus \cite[Chapter X \S102]{Om}.} spinor genus.

\end{theorem}

\begin{proof}

See \cite[\S11, Corollary pg 213]{cassels}. 
 
\end{proof}

\subsection{Outline of the proof}

Let $K$ be a number field with maximal order $O_{K}$  and let $q_{K}$ be its integral trace form i.e., the integral quadratic form obtained by restricting 
$x \mapsto \mathrm{tr}_{K/\Q}(x^{2})$ to $O_{K}$. The maximal order $O_{K}$  is an integral lattice, with integral quadratic $q_{K}$, 
in the quadratic space $(K, \rm{tr}_{K/\Q}(x^2))$. We will denote this lattice and quadratic space by $\Lambda$ and $V$, respectively. To obtain that 
${\rm spin}^{+}(\Lambda) = {\rm gen}(\Lambda)$ it is enough to show, thanks to Theorem \ref{mainhelp}, that for all prime $p$  
\begin{equation}\label{contiene}
\theta_{p}(\Lambda_{p}) \supseteq \Z_{p}^{*} \bmod (\Q_{p}^{*})^2.
\end{equation}

For primes $p$ that are at worst tamely ramified in $K$ there is a Jordan decomposition of the lattice $\Lambda_{p}$ that allows us to verify (\ref{contiene}); see Theorem \ref{TeoA} for details. \\

To obtain (\ref{contiene}) for wild primes we use the following result together with standard upper bounds on the $p$-adic valuations of the discriminant of a number field (see Corollary \ref{property}).

\begin{theorem}\label{Ebounds}
 Let $V$ be a non-degenerate rational quadratic space of dimension $n >2$. Let $\Lambda$ be a full rank integral lattice of discriminant $D$, 
and suppose that $p$ is a prime such that $\theta_{p}(\Lambda_{p}) \not\supseteq \Z_{p}^{*} \bmod (\Q_{p}^{*})^2$. Then, 
  
  \begin{enumerate}
  
  \item $n(n-1)/2 \leq v_{p}(D)$ if $p$ is odd. 
  
  \item $n(n-3) + 2\left \lfloor \frac{n+1}{2} \right \rfloor \leq v_{2}(D)$ if $V$ is definite.
  
  \item $n(n-1) \leq v_{2}(D)$ if $V$ is indefinite.
  
  \end{enumerate}
  
\end{theorem}

\begin{proof}

See both Corollaries in page 214 of \cite{cassels} and the proofs of \cite[Theorem 4.2, Theorem 4.6]{ernts}.
 
\end{proof}

\begin{remark}
It is worthwhile to point out that the initial step in the strategy to show (\ref{contiene}) is the same for every prime regardless of the ramification type. We explain this by looking at the argument for an odd prime $p$. The common strategy here is to use a diagonalization $\langle \alpha_{1},...,\alpha_{n} \rangle$ of  the $\Z_{p}$-lattice $\Lambda_{p}$ and obtain (\ref{contiene}) from information about the $\alpha_{i}$'s. In the case of at worst tame ramification we know explicitly the values of the $\alpha_{i}$'s and from them we get (\ref{contiene}) (see Theorem \ref{TeoA}). On the other hand, if the prime $p$ is wild we can assume that $v_{p}(\alpha_{i})\neq v_{p}(\alpha_{j})$ for $i \neq j$, otherwise $\Lambda_{p}$ contains an orthogonal factor $\Lambda^{(p)}$ with $\theta_{p}(\Lambda_{p}) \supseteq \theta_{p}(\Lambda^{p})  \supseteq \Z_{p}^{*} \bmod (\Q_{p}^{*})^2$. In particular, we have that \[v_{p}(\alpha_{1}\cdot...\cdot\alpha_{n}) = \sum_{i=0}^{n}v_{p}(\alpha_{i}) \ge n(n-1)/2.\] This is precisely the idea behind Theorem \ref{Ebounds}.1, hence to obtain (\ref{contiene}) we are only left with the task to show that this inequality never occurs for $\langle \alpha_{1},...,\alpha_{n} \rangle$. 
\end{remark}

\section{The proof}  

As often happens, it is convenient to start with number fields with only tame ramification. Afterwards we will proceed to the case of wild ramification. Throughout this section all number fields are assumed to have degree bigger than $2$.

\subsection{Local integral trace at a tame prime}

In this section we describe the Jordan decomposition of the integral trace form when localized at a prime with at worst tame ramification. For details and proofs see \cite{Manti3}.  Given $L$ a number field and $p$ a rational prime we denote by  $g_{p}$ the number of primes in $O_{L}$ lying over $p$. Furthermore, 
\[ F_{p}:=\sum_{i=1}^{g_{p}}f_{i}\] where $f_{1},..., f_{g_{p}}$ are the inertia degrees of the prime $p$ in $L$.
\subsubsection{Jordan decomposition}

The following result describes the Jordan decomposition of the 
localization of the integral trace in terms of residue and inertia degrees. For details see \cite{Manti3}.

\begin{theorem}\label{general}
Let $L$ be a degree $n$ number field and let $q_{L}$ be its integral trace form.  Let $p$ be a rational prime which is not wildly ramified in $L$. Then,
\[q_{L} \otimes \Z_{p} \cong 
\begin{cases}
\underbrace{\langle1,....,1,\alpha_{p} \rangle}_{F_{p}}  \bigoplus p \otimes \underbrace{\langle 1,...,1,\beta_{p} \rangle}_{n-F_{p}}  \mbox{ if $p \neq 2$}, \\
\underbrace{\langle\gamma_{1},....,\gamma_{F_{2}} \rangle}_{F_{2}}  \bigoplus 2 \otimes \underbrace{( \mathbb{H} \oplus...\oplus \mathbb{H})}_{\frac{n-F_{2}}{2}} \ \mbox{ otherwise.}\\
\end{cases}\]
Where $\alpha_{p}, \beta_{p}, \gamma_{i}$ are some elements  $\Z_{p}^{*}$ that can be explicitly calculated in terms of the ramification of the prime $p$ at $L$.

\end{theorem}

\begin{proof}
See \cite[Theoreme 1.3]{emp} and \cite[Theorem 0.1]{Manti3}.
\end{proof}

\subsubsection{Tame primes}  

Thanks to Theorem \ref{general} we can easily deal with primes that at worst tamely ramified. Recall that we denoted by $V$ the quadratic space  given by $(K, \rm{tr}_{K/\Q}(x^2))$, where $K$ is number field of dimension at least $3$, and by $\Lambda$ the integral lattice in $V$ given by the maximal order $O_{K}.$

 \begin{theorem} \label{TeoA}
Let $p$ be a prime, and suppose that $p$ is unramified or tame in $K$. Then, $\theta_{p}(\Lambda_{p}) \supseteq \Z_{p}^{*} \bmod (\Q_{p}^{*})^2$.

\end{theorem}

\begin{proof}

Suppose first that $p \neq 2$. By Theorem \ref{general} there are $\alpha, \beta \in \Z_{p}^{*}$ such that  \[\Lambda_{p} \cong \langle 1,1,..., 1,\alpha \rangle \bigoplus \langle p,p,..., p\beta \rangle \]  Since ${\rm dim} \Lambda_{p} \ge 3$ one of the factors in the above orthogonal decomposition has dimension at least $2$. Hence $\Lambda_{p}$ contains a two dimensional orthogonal factor $\Lambda^{(p)}$ among the following $\langle 1, \alpha \rangle,  \langle p, p\beta \rangle$. If $p=2$ then, again by Theorem \ref{general}, there are $\alpha_{1},..., \alpha_{m} \in \Z_{2}^{*}$ such that  
\[\Lambda_{2} \cong  \langle \alpha_{1},..., \alpha_{m} \rangle\bigoplus 2 \otimes ( \mathbb{H} \oplus...\oplus \mathbb{H}).\]  Thus $\Lambda_{2}$ contains an orthogonal factor $\Lambda^{(2)}$ of the form $2\mathbb{H}$ or $\langle \alpha_{1}, \alpha_{2}, \alpha_{3} \rangle.$ The result follows from \cite[XI, \S3, Lemma 3.7, Lemma 3.8]{cassels}.

\end{proof}

\subsection{Wild primes} 

We first recall some classic bounds on the $p$-adic valuation of the discriminant of a number field.

\paragraph{Discriminant bounds.} 
  
Let $n$ be the degree of $K$ and Let $D_{K}$ be its discriminant. Let $p$ be a prime and suppose that $e_{1},...,e_{g}$ are the ramification degrees of $p$ with corresponding residue degrees $f_{1},...,f_{g}$. Let 
$\displaystyle F_{p}:= \sum_{i=1}^{g}f_{i}$ and $M_{p}=\max\{v_{p}(e_{i})\}$. The following well known bounds on the discriminant can be found in \cite[III, \S6, Remark Proposition 13.]{Serre2}.

\begin{theorem}\label{discbound}

\[v_{p}(D_{K}) \leq n-F_{p}+\sum_{i=1}^{g}e_{i}f_{i}v_{p}(e_{i}),\] where equality is obtained if $M_{p}=0$.
\end{theorem}

\begin{remark}
Notice that $M_{p}=0$ is another way to say that the prime $p$ is either unramified or tame in $K$. In such case $v_p(D_{K})=n-F_{p}$ follows immediately from Theorem \ref{general}.
\end{remark}
An immediate consequence of Theorem \ref{discbound} is:
\begin{corollary}\label{property}
\[v_{p}(D_{K}) \leq n(M_{p}+1)-F_{p}.\]

\end{corollary}

\paragraph{Odd primes.} 

In the case of wild ramification it is convenient to separate the cases $p=2$ and $p\ne 2.$ Recall that in this section we are always assuming that that $n \ge 3$.

\begin{theorem}\label{over5}
Let $p \ge 5$ be a prime. Then,  $\theta_{p}(\Lambda_{p}) \supseteq \Z_{p}^{*} \bmod (\Q_{p}^{*})^2$. 
\end{theorem}

\begin{proof}
If $n <5$ then $p$ is either unramified or tame, both cases covered in Theorem \ref{TeoA}. Hence we may assume that $5 \leq n$. Under this assumption we have that $n \leq 5^{\frac{n-3}{2}} \leq  p^{\frac{n-3}{2}}$. Since
$p^{M_{p}} \leq n$ it follows that $M_{p} \leq \frac{n-3}{2}$, thus $v_{p}(D_{K}) \leq n(n-1)/2-F_{p}< n(n-1)/2.$ The result follows from Theorem \ref{Ebounds}.(1)

\end{proof}

\begin{theorem}\label{tres}
$\theta_{3}(\Lambda_{3}) \supseteq \Z_{3}^{*} \bmod (\Q_{3}^{*})^2$. 
\end{theorem}

\begin{proof}

It is convenient to divide the problem in cases according to the degree $n$. As in the proof of Theorem \ref{over5} we may assume that $3$ 
has wild ramification in $K$ i.e., $1\leq M_{3}$.\\

\begin{itemize}

\item If $5 \leq n \leq 8$ then $M_{3}=1$. Thus, $M_{3} \leq \frac{n-3}{2}$ which leads to the result in the exact same way as in the proof 
of Theorem \ref{over5}.\\

\item If $9 \leq n$ then $n < 3^{\frac{n-3}{2}}$. Since $3^{M_{3}} \leq n$ we see that $M_{3} < \frac{n-3}{2}$ and we argue as before.\\

\item Suppose $n=3$. Since $3$ is ramified then tr$_{K/\Q}(O_{K}) \subseteq 3\Z$. Thus, after diagonalizing $\Lambda$ over $\Z_{3}$ we get that $\Lambda_{3} \cong \langle 3^{a_{1}}u_{1}, 3^{a_{2}}u_{2}, 3^{a_{3}}u_{3} \rangle$ where each $u_{i} \in \Z_{3}^{*}$ and the $a_{i}$'s are positive integers. Notice that $3$ is totally ramified in $K$, hence by Theorem \ref{discbound} we have that $v_{3}(D_{K}) \leq 5$. It follows that there are $i,j$ with $i \neq j$ and such that $a_{i}=a_{j}$, otherwise $v_{3}(D_{K}) =a_{1}+a_{2}+a_{3} \ge 1+2+3=6.$ Thus $\Lambda_{3}$ contains a two dimensional orthogonal factor $\Lambda^{(3)}$ of the form $\langle 3^{a}, 3^{a}u \rangle$, where $a$ is a positive integer and $u \in \Z_{3}^{*}$, and the result follows from \cite[XI, \S3, Lemma 3.7, Lemma 3.8.]{cassels}.\\  

\item Suppose $n=4$. Since $3$ is ramified then $3O_{K}=\mathcal{B}_{1}^{3}\mathcal{B}_{2}$ for some prime ideals $\mathcal{B}_{i}$'s of $O_{K}$. It follows that $\Lambda_{3} \cong \langle 1 \rangle \oplus \Lambda^{'}_{3}$
where $\Lambda^{'}_{3}$ is the integral trace lattice of a totally ramified cubic extension of $\Q_{3}$. By the same argument we used in the case $n=3$ applied to $\Lambda^{'}_{3}$ we conclude that $\Lambda^{'}_{3}$  has an orthogonal factor  $\Lambda^{(3)}$ of the form $\langle 3^{a}, 3^{a}u \rangle$, and thus it is also an orthogonal factor of $\Lambda_{3}$ from which the result follows.
\end{itemize}

\end{proof}

\paragraph{The case $p=2$.}
Before dealing with the even prime we need an auxilliary result.
\begin{lemma}\label{lema2} 
Suppose $n \neq 4$. Then, \[v_{2}(D_{K}) \leq n(n-2)\] with equality possible only in the case $n=3$.
\end{lemma}

\begin{proof}
If $n\geqslant 6$ then $2^{n-3}>n \ge 2^{M_{2}}$, thus $M_{2} < n-3$. Since $2^{M_{2}} \leq n$ we have that $M_{2} \leq 2$ if $n=5$. In particular, for all $n\geqslant 5$ we have that $M_{2} \leq n-3.$ Thanks to Corollary \ref{property} we have that $v_{2}(D_{K}) \leq n(n-2) -F_{2} < n(n-2)$. Now suppose $n=3.$ If $2$ is unramified the inequality is obvious. Otherwise $2$ has ramification degrees $1,2$ and residue degrees $1,1$ which by Theorem \ref{discbound} implies that  $v_{2}(D_{K}) \leq 3 =n(n-2).$
\end{proof}

\begin{proposition}\label{dosnonreal}
Suppose that $K$ is not a totally real quartic field. Then, $\theta_{2}(\Lambda_{2}) \supseteq \Z_{2}^{*} \bmod (\Q_{2}^{*})^2$. 
\end{proposition}

\begin{proof}

Since $n(n-2) \leq \min\{n(n-3) + 2\left \lfloor \frac{n+1}{2} \right \rfloor , n(n-1)\}$ the result follows, in the case $n \neq 4$, from Lemma \ref{lema2} and Theorem \ref{Ebounds}. If $n=4$, we have that $M_{2} \leq 2$ and by Corollary \ref{property} we see that  $v_{2}(D_{K}) \leq 11$. On the other hand since $K$ is  non totally real $\Lambda$ is indefinite and the result follows from Theorem \ref{Ebounds}.(3).  

\end{proof}

\begin{remark}
Most of the hard work in proving the main result comes when dealing with the primes $2$ and $3$ for small values of $n$. In particular notice that at this point we can already conclude that  \[{\rm gen}(q_{K})={\rm spin}^{+}(q_{K})\] for all number fields $K$ that are either at worst tame at $2$ or that are non-quartic and totally real. Next we deal with this, the final case.
\end{remark}

\begin{theorem}\label{dos}
$\theta_{2}(\Lambda_{2}) \supseteq \Z_{2}^{*} \bmod (\Q_{2}^{*})^2$.

\end{theorem}

\begin{proof}

We may assume that $n=4$ and that $M_{2} \in \{1,2\}$. If $M_{2}=1$ then, by  Corollary \ref{property}, $v_{2}(D_{K}) \leq 7 <n(n-2)$ and we use the argument of Proposition \ref{dosnonreal}. Furthermore, if $c:=v_{2}(D_{K}) \leq 8$ the same argument applies hence we may assume that $M_{2}=2$ and that $9 \leq c$. Therefore $L:=K\otimes_{\Q}\Q_{2}$ is a totally ramified quartic extension of $\Q_{2}$, with discriminant of valuation $ c \in \{9,10,11\}$. There are 36 such fields (see for instance \cite{Jones}).The ring of integers of $L$ is of the form $O_{L}=\Z_{2}[\alpha_{L}]$ for some in $\alpha_{L} \in O_{L}$. The table below consists of the minimal polynomial of $\alpha_{L}$ for each one of the 36 fields and their corresponding discriminant valuation $c$.

\begin{center} 
 \begin{tabular}{|c|c|c|c|c|} \hline
$c=9$ & $c=10$ & $c=11$ & $c=11$ & $c=11$ \\  \hline
$x^4+6x^2+2$& $x^4+2x^2-9$ & $x^4+12x^2+2$ & $x^4+6$ & $x^4+8x+10$  \\ \hline
$x^4-2x^2+2$& $x^4+2x^2-1$ & $x^4+4x^2+18$ & $x^4+22$ & $x^4+8x+6$ \\ \hline
$x^4+6x^2+10$& $x^4+6x^2-9$ & $x^4+12x^2+18$ & $x^4+14$ & $x^4+8x+14$ \\ \hline
$x^4+2x^2+10$& $x^4+6x^2-1$ & $x^4+4x^2+10$ & $x^4+30$ & $x^4+8x^2+8x+22$ \\ \hline
$x^4+2x^2-2$& $x^4-6x^2+3$ & $x^4+12x^2+10$ & $x^4+26$ &   \\ \hline
$x^4-2x^2-2$& $x^4+6x^2+3$ & $x^4+4x^2+14$ & $x^4+10$ &  \\ \hline
$x^4+2x^2+6$& $x^4-2x^2+3$ & $x^4+4x^2+6$ & $x^4+18$ &  \\ \hline
$x^4-2x^2+6$& $x^4+2x^2+3$ & $x^4+12x^2+6$ & $x^4+2$ &  \\ \hline 
 \end{tabular}
\end{center}

\begin{itemize}
\item[$\bullet$] Suppose that the minimal polynomial of $\alpha_{L}$ is of the form $x^{4}+2ax^2+b$.\\ 
Out of the 36 polynomials in the above table, 32 are of this form and they satisfy exactly one of the following
\begin{enumerate}
\item[(i)] $v_{2}(a)=0, v_{2}(b)=1$, \item[(ii)] $v_{2}(a)=0, v_{2}(b)=0$, \item[(iii)] $v_{2}(a)=1, v_{2}(b)=1$. \item[(iv)] $a=0, v_{2}(b)=1$. 
\end{enumerate} In particular, unless $a=0$ i.e, in case (iv), $b/a \in  \Z_{2}$.
Moreover, in case (ii) we have that all the values of $b$ satisfy that $b \equiv -1 \bmod 4.$ The Gram matrix of the trace form in the basis $\{1, \alpha_{L}, \alpha_{L}^{2}, \alpha_{L}^{3}\}$ is given by

$\begin{bmatrix}  4 & 0 & -4a & 0 \\ 0 &  -4a & 0 & 8a^{2}-4b  \\ -4a & 0 & 8a^{2}-4b & 0 \\ 0 & 8a^{2}-4b& 0 & -16a^{3}+12ab  \end{bmatrix} 
\cong\footnote{By this we mean that the quadratic forms having such Gram matrices are equivalent.}  
\begin{bmatrix}  4 & 0 & 0 & 0 \\ 0 &  -4a & 0 & 8a^{2}-4b  \\ 0 & 0 & 4a^{2}-4b & 0 \\ 0 & 8a^{2}-4b& 0 & -16a^{3}+12ab \end{bmatrix}$\\ 
$\cong \begin{bmatrix}  4 & 0 & 0 & 0 \\ 0 &  4a^{2}-4b & 0 & 0  \\ 0 & 0 & -4a & 8a^{2}-4b \\ 0 & 0 & 8a^{2}-4b & -16a^{3}+12ab \end{bmatrix} 
\cong \begin{bmatrix}  4 & 0 & 0 & 0 \\ 0 &  4a^{2}-4b & 0 & 0  \\ 0 & 0 & -4a & -4b \\ 0 & 0 & -4b & -4ab  \end{bmatrix}$.\\ 

The last form is $\Z_{2}$-isomorphic to \[  
 \begin{bmatrix}  4 & 0 & 0 & 0 \\ 0 &  4a^{2}-4b & 0 & 0  \\ 0 & 0 & -4a & 0 \\ 0 & 0 & 0 & -4\frac{b}{a}(a^2-b) \end{bmatrix} \quad \mbox{for $a \neq 0$} \] and to \[\begin{bmatrix}  4 & 0 & 0 & 0 \\ 0 &  -4b & 0 & 0  \\ 0 & 0 & 0 & -4b \\ 0 & 0 & -4b & 0  \end{bmatrix} \quad \mbox{otherwise}.\]

It follows that

\[\Lambda_{2} \cong \begin{cases} \left \langle 2^2,2^{2}(a^2-b),2^{2}(-a),2^{3}(\frac{a^{2}b-b^2}{2a}) \right \rangle & \mbox{ if $v_{2}(a)=0$ and $v_{2}(b)=1$}, \\ \left \langle 2^2,2^{3}\frac{(a^2-b)}{2},2^{2}(-a),2^{3}(\frac{a^{2}b-b^2}{2a}) \right \rangle & \mbox{ if $v_{2}(a)=0$ and $v_{2}(b)=0$},\\ \left \langle 2^2,2^{3}\frac{(a^2-b)}{2},2^{3}(\frac{-a}{2}),2^{3}(\frac{a^{2}b-b^2}{2a}) \right \rangle & \mbox{ if $v_{2}(a)=1$ and $v_{2}(b)=1$},  \\  \left \langle 2^2,2^{3}(\frac{-b}{2}) \right \rangle \oplus 2^{3}\mathbb{H} & \mbox{ if $a=0$ and $v_{2}(b)=1$}. \end{cases}\]

Notice that in case (i), since $v_{2}(a^2) < v_{2}(b)$, we have that $v_{2}(a^2-b)=v_{2}(a^2)=0$ i.e., $a^2-b \in \Z_{2}^{*}.$ Similarly in case (iii) it follows from $v_{2}(b) < v_{2}(a^2)$ that $\frac{(a^2-b)}{2} \in \Z_{2}^{*}$.  In case (ii) we also have that $\frac{(a^2-b)}{2} \in \Z_{2}^{*}$; this follows since in this case all the values of $b$ satisfy that $b \equiv -1 \bmod 4.$   

\item[$\bullet$] If the minimal polynomial of $\alpha_{L}$ is of the form $x^{4}+8x+2b$ with $v_{2}(b)=0$, three out of the remaining four polynomials in the table are of this form, then similarly as above we have that \[ \Lambda_{2} \cong \left \langle 2^2,2^{3}(-b) \right \rangle \oplus 2^{3}\mathbb{H}.\]

\item[$\bullet$] Finally if the minimal polynomial of $\alpha_{L}$ is $x^4+8x^2+8x+22$ then \[ \Lambda_{2} \cong \left \langle 2^2,2^{3}(-3), 2^{3}, 2^{3}\left ( \frac{-773}{3} \right) \right \rangle. \] 

\end{itemize}
 From the description above we see that $\Lambda_{2}$ contains an orthogonal factor $\Lambda^{(2)}$ of one of the following forms:  $2^{3}\mathbb{H}$ or $\langle 2^eu_{1}, 2^eu_{2}, 2^eu_{3}\rangle$ or $\langle 2^eu_{1}, 2^eu_{2}, 2^{e+1}u_{3}\rangle$, where all the $u_{i}'$s are in $\Z_{2}^{*}$.  The result follows from \cite[XI, \S3, Lemma 3.8]{cassels}.

\end{proof}

\begin{theorem}\label{principal}

Let $K$ be a number field of degree at least $3$. Then, the genus of integral trace form $q_{K}$ contains only one proper spinor genus, thus it only contains one spinor genus.

\end{theorem}

\begin{proof}

This follows from Theorems \ref{over5}, \ref{tres}, \ref{dos} and \ref{mainhelp}.

\end{proof}

\begin{corollary}
Let $K$ be a non totally real number field. Then, any integral quadratic form that is equivalent to $q_{K}$ is properly equivalent to it.
\end{corollary}

\begin{proof}
The result is clear for $K$ quadratic, just compose with complex conjugation on $K$ whenever the original isometry is not proper. For higher dimensions  we have \[\begin{matrix} {\rm cl}(q_{K}) & \subseteq^{1} & {\rm spin}(q_{K}) & \subseteq & {\rm gen}(q_{K}) \\ \rotatebox[origin=c]{270}{ $\supseteq$ }  & & \rotatebox[origin=c]{270}{ $\supseteq$ } & & \rotatebox[origin=c]{270}{ $=$ }\\  {\rm cl}^{+}(q_{K}) & \subseteq^{2} & {\rm spin}^{+}(q_{K}) & \subseteq^{3} & {\rm gen}^{+}(q_{K}). \end{matrix}\] where contaiments $1$ and $2$ are equalities by the celebrated result of Eichler \cite{eichler} and containment $3$ is an equality by Theorem \ref{principal}. It follows that all the containments in the above diagram are equalities, in particular we have that ${\rm cl}^{+}(q_{K}) ={\rm cl}(q_{K}).$ 

\end{proof}

\section{Quadratic fields}

Our interest in the spinor genus of the trace began as the search for a tool to understand when a pair of number fields have the same integral 
traces. Since quadratic fields are completely characterized by their discriminants the question about the isometry between their traces becomes 
trivial. However, since all the methods applied above relied heavily on the assumption that the number fields have dimension at least 
$3$, it is interesting on its own to see if Theorem \ref{principal} remains valid in dimension $2$. The following example shows that this is not always 
the case. 

\begin{example}
Let $K$ be the quadratic number field of discriminant $17$. Then, its integral  trace form $q_{K}$ is equivalent to the form 
$\langle 2,2,9 \rangle$ which genus contains two spinor genus; $\langle 2,2,9\rangle$ and $\langle1, 17 \rangle$.
\end{example}

\begin{remark}
The above example is minimal in the sense that for any quadratic field $K$ with discriminant $d$ with $|d| <17$ one has that ${\rm spin}(q_{K}) ={\rm gen}(q_{K}).$ 
\end{remark}

It is natural to wonder for which quadratic fields $K$ the conclusion of Theorem \ref{principal} does hold, and for which it does not. A simple way to construct 
quadratic fields that satisfy Theorem \ref{principal} comes from Gauss' genus theory(see \cite[II \S6,\S7]{cox}). Suppose that $\Delta$ is an integer of the form $\Delta=f^2D$, where $f$ is a positive integer and $D$ is a {\it fundamental discriminant} i.e., the discriminant of a quadratic number field. We will denote by $C_{\Delta}$ the {\it narrow class group} of the order of conductor $f$ in $\Q(\sqrt{D})$.

\begin{theorem}[Gauss] Let $\Delta$ be an integer as above. The set of integral primitive binary quadratic forms of discriminant $\Delta$, under proper equivalence, has an structure of abelian group isomorphic to $C_{\Delta}$. Moreover the set of genus forms of discriminant $\Delta$ is also a group, and it isomorphic to $C_{\Delta}/C^{2}_{\Delta}.$
\end{theorem}  
\noindent Let $K$ be a quadratic number field and let us say for simplicity that the discriminant $d$ of $K$ is odd. The trace form $q_{K}$ is an integral primitive binary quadratic form of discriminant $-4d$ (see lemma below). Notice that $-4d$ is also a fundamental discriminant. If the narrow class group $C_{-4d}$  has exponent at most $2$ then, by Gauss' theorem, for every 
primitive binary quadratic form $q$ of discriminant $-4d$ we have that ${\rm cl^{+}}(q) ={\rm gen}(q)$ thus ${\rm spin^{+}}(q) ={\rm gen}(q).$ In particular, if $C_{-4d}$ has exponent at most $2$ then \[{\rm spin^{+}}(q_{K}) ={\rm gen}(q_{K})\] for every quadratic field of discriminant $d$. Examples of this are given by the quadratic fields $\Q(\sqrt{-3})$, $\Q(\sqrt{5})$ and $\Q(\sqrt{13})$. Using the following lemma one can also construct examples of quadratic fields $K$ with even discriminant for which ${\rm spin}(q) ={\rm gen}(q)$ e.g., $K=\Q(i)$.

Given a square free integer $d$ we let $K_{d}=\Q(\sqrt{d})$.
 \begin{lemma}\label{tracequadratic}
Let $d\neq 1$ be a non-zero square free integer. If $d \equiv 1\pmod{4}$ then the trace form $q_{K_{d}}$ is an integral primitive binary quadratic form of discriminant $-4d$. Otherwise, the form $\frac{1}{2}q_{K_{d}}$ is an integral primitive binary quadratic form of discriminant $-4d$.
\end{lemma}

\begin{proof}
Using the usual integral basis for $O_{K_{d}}$  we see that the integral trace form can be written in the form \[ q_{K_{d}} \cong \left \langle 2,2,\frac{1+d}{2} \right \rangle\] whenever $d \equiv 1\pmod{4}$ or as $\displaystyle q_{K_{d}} \cong \langle 2,0, 2d \rangle$ whenever $d \equiv 2,3\pmod{4}$.
\end{proof}

To find sufficient and necessary conditions on a quadratic field $K$ to see when the conclusion of Theorem \ref{principal} is valid over $K$ Gauss' genus theory is not enough. However, the following generalization of Estes and Pall of Gauss' result, see \cite[Corllary 1]{EstesPall}, gives a sufficient criterium to find such $K$'s. 

\begin{theorem}[Estes,Pall] \label{estespall} Let $\Delta$ be an integer as above. The set of proper spinor genera of integral primitive binary quadratic forms of discriminant $\Delta$ forms a group isomorphic to $C_{\Delta}/C^{4}_{\Delta}.$ Moreover, the set of proper spinor genera in a genus is also a group isomorphic to $C^{2}_{\Delta}/C^{4}_{\Delta}.$
\end{theorem} Simply put, proper spinor genus and genus coincide as long as $C_{\Delta}$ does not have elements of order $4$ i.e., as long as $C_{\Delta}$ has trivial $4$-rank.

\begin{definition}
Let $G$ be a finitely generated abelian group. Let $p$ be a prime and let $n$ be a positive integer. The {\it $p^{n}$-rank of the group $G$} is the non-negative integer defined by \[{\rm rk}_{p^{n}}(G)={\rm dim}_{\mathbb{F}_{p}}(G^{p^{(n-1)}} \otimes_{\Z} \mathbb{F}_{p}).\]  
\end{definition}

\begin{proposition}\label{4ranktheorem}
Let $d\neq 1$ be a non-zero square free integer and let $G_{d}$ be the narrow class group $C_{-4d}$. Then, \[{\rm spin^{+}}(q_{K_{d}}) ={\rm gen}(q_{K_{d}}) \Longleftrightarrow {\rm rk}_{4}(G_{d})=0.\]
\end{proposition}

\begin{proof}
This follows clearly from Lemma \ref{tracequadratic} and Theorem \ref{estespall}.
\end{proof}

\begin{corollary}
There are infinitely many quadratic fields, in fact a positive proportion of them, that do not satisfy the conclusion of Theorem \ref{principal}. 
\end{corollary}

\begin{proof}
Let $p$ be a prime such that $p \equiv 1 \pmod 8$, and let $K_{p}=\Q(\sqrt{p})$. Since \[{\rm rk}_{4}(G_{p})\ge 1\] (see \cite[Proposition 2]{Sound}) we have by Proposition \ref{4ranktheorem} that  \[{\rm spin^{+}}(q_{K_{p}}) \neq {\rm gen}(q_{K_{p}}).\] \end{proof}

Similarly it can be shown that there is a positive proportion of quadratic fields for which Theorem \ref{principal} is valid. An interesting problem arising from this is whether such proportions can be made explicit. We formalize this question as follows: Let $T(X)$ be the number of square free integers $d$ such that $|d| <X$ and such that the conclusion of Theorem \ref{principal} is valid for $K_{d}$. Let $N(X)$ be the number of square free integers $d$ with $|d| <X$.

\begin{question}
Does the limit \[\lim_{X \to \infty }\frac{T(X)}{N(X)}\] exist? and if so, what is its value?
\end{question}

Going further one could even ask what proportion of real, resp. complex, quadratic fields satisfy the conclusion of Theorem \ref{principal}. We answer all these questions explicitly:

\begin{theorem}\label{teoremacaso2} Let $T^{\pm}(X)$ be the number of square free integers $d$ such that $0<\pm d <X$ and such that the conclusion of Theorem \ref{principal} is valid for $K_{d}$. The values $N^{\pm}(X)$ are defined in a similar fashion. The following limits exist \[ \alpha=\lim_{X \to \infty }\frac{T(X)}{N(X)} \ , \alpha^{+}=\lim_{X \to \infty }\frac{T^{+}(X)}{N^{+}(X)}  \ \mbox{and} \ \alpha^{-}= \lim_{X \to \infty }\frac{T^{-}(X)}{N^{-}(X)}.\] Furthermore, if $\phi(q)$ denoted the Euler $q$-series $\displaystyle \prod_{n \ge 1}(1-q^{n})$ then \[ \alpha^{+}=\phi(1/2) \approx 0.29, \ \alpha^{-}=2\phi(1/2) \approx 0.58\ \mbox{and} \ \alpha= \frac{\alpha^{+}+ \alpha^{-}}{2}\approx 0.43.\]
\end{theorem}
\noindent In other words, $43\%$ (resp. $29\%$, resp. $58\%$) of quadratic (resp. real quadratic, resp. imaginary quadratic) fields  satisfy the conclusion of Theorem \ref{principal}.

\subsection{Proof of Theorem \ref{teoremacaso2}}
Before proving Theorem \ref{teoremacaso2} we need some preliminary results. The main tool behind the above theorem is the recent work of Fouvry and Kl\"uners. For details see \cite{FK}, and more specifically \cite[Theorem 1.1]{FK2}.
\begin{theorem}[Fouvry, Kl\"uners]\label{FK}
Suppose that in the following $D$ runs over fundamental discriminants. Then,
\begin{align*}
& \left | \{0< -D < X,  D \equiv 12 \bmod{16}, \ {\rm rk}_{4}(C_{D})=0 \} \right | &=& \ {\phi(1/2)\left(\frac{1}{2\pi^2}X +o(X)\right)} \\
& \left | \{ 0<  D < X,  D \equiv 12 \bmod{16}, \ {\rm rk}_{4}(C_{D})=0 \} \right | &=& \ {2\phi(1/2)\left(\frac{1}{2\pi^2}X +o(X)\right)}\\
& \left | \{ 0< -D < X,  D \equiv 8 \bmod{16}, \ {\rm rk}_{4}(C_{D})=0 \} \right | &= & \ {\phi(1/2)\left(\frac{1}{2\pi^2}X +o(X)\right)}\\
& \left | \{ 0<  D < X,  D \equiv 8 \bmod{16}, \ {\rm rk}_{4}(C_{D})=0 \} \right | &=& \ {2\phi(1/2)\left(\frac{1}{2\pi^2}X +o(X)\right)} \\
& \left | \{ 0< -D < X,  D \equiv 1 \bmod{4}, \ {\rm rk}_{4}(C_{D})=0 \} \right | &=& \ {\phi(1/2)\left(\frac{2}{\pi^2}X +o(X)\right)}\\
& \left | \{ 0<  D  < X,  D \equiv 1 \bmod{4}, \ {\rm rk}_{4}(C_{D})=0 \} \right | &=& \ {2\phi(1/2)\left(\frac{2}{\pi^2}X +o(X)\right).}
\end{align*}
\end{theorem}

Given $K$ a number field and $O$ an order in it we denote by $P(O)$ (resp. $P^{+}(O)$) the group of principal fractional $O$-ideals (resp. the group of of totally positive principal fractional $O$-ideals).

\begin{lemma}\label{D}
Let K be a real quadratic number field of odd discriminant and let $O$ be the order of conductor $2$ in $O_{K}$. Then the natural map  \begin{eqnarray*} \alpha: P(O)/P^{+}(O) & \to & P(O_{K})/P^{+}(O_{K}) \\ \ class(x O) & \mapsto & class(x O_{K}) \end{eqnarray*} is an isomorphism.
\end{lemma}

\begin{proof}
Let $\widetilde{O}$ be an order in $K$ and let $\epsilon_{\widetilde{O}}$ be its fundamental unit. The group $P(\widetilde{O})/P^{+}(\widetilde{O})$ is either trivial or isomorphic to the group of order $2$, and the former happens if and only if the norm of the fundamental unit $N(\epsilon_{\widetilde{O}})$ is negative. Since $\alpha$ takes the class of $\epsilon_{O}$ to the class of $\epsilon^{m}_{O_{K}}$, where $\displaystyle m:=\#(O^{*}_{K}/O^{*})$, it follows from the initial observation that it is enough to show that $m$ is odd. Let $\mathcal{F}:=\{\gamma \in O_{K}: \gamma O_{K} \subseteq O\} $ i.e, the conductor of $O$, which in this case is $2O_{K}$. Since $2$ is unramified in $K$ the group $(O_{K}/\mathcal{F})^{*}$ is either trivial of order $3$, hence the same is true for its quotient group $\displaystyle (O_{K}/\mathcal{F})^{*}/(O/\mathcal{F})^{*}$. Since the group $(O^{*}_{K}/O^{*})$ can be injected in the group $\displaystyle (O_{K}/\mathcal{F})^{*}/(O/\mathcal{F})^{*}$, see for instance \cite[I, \S12 Theorem 12]{neu}, we have that $m \mid 3$ and in particular $m$ is odd.
\end{proof}

Given $K$ a number field and $O$ an order in it we denote by $Pic(O)$ (resp. $Pic^{+}(O)$) the Picard group of $O$ (resp. the narrow Picard group of $O$). If $O$ is the maximal order of $K$ we denote $Pic(O)$ (Resp. $Pic^{+}(O)$) by $C\ell(K)$ (Resp. $C\ell^{+}(K)$.)

\begin{proposition}\label{D2}
Let $K$ be a real quadratic number field of odd discriminant and let $O$ be the order of conductor $2$ in $O_{K}$. Then $Pic(O)[2^{\infty}]\cong C\ell(K)[2^{\infty}]$ if and only if $Pic^{+}(O)[2^{\infty}] \cong C\ell^{+}(K)[2^{\infty}]$. In particular, \[{\rm rk}_{4}(Pic^{+}(O))={\rm rk}_{4}(C\ell^{+}(K))\] whenever $Pic(O)[2^{\infty}]\cong C\ell(K)[2^{\infty}]$ .

\end{proposition}

\begin{proof}
Consider the following commutative diagram: (here $\beta$ and $\gamma$ are defined similarly to $\alpha$ in Lemma \ref{D} ) 

\[\xymatrix{1 \ar[r] & P(O)/P^{+}(O) \ar[d]^{\alpha} \ar[r] & Pic^{+}(O) \ar[d]^{\beta} \ar[r] & Pic(O) \ar[d]^{\gamma} \ar[r] & 1 \\
1 \ar[r] & P(O_{K})/P^{+}(O_{K}) \ar[r] & C\ell^{+}(K) \ar[r] & C\ell(K)  \ar[r] & 1}\]
 It follows from the Snake lemma and from Lemma \ref{D} that $\ker(\beta) \cong \ker(\gamma)$ and that ${\rm Coker}(\beta) \cong {\rm Coker}(\gamma)$. On the other hand since $\gamma$ is surjective (see\cite[I, \S12 Proposition 9]{neu}), $\beta$ is surjective as well. Since all the groups involved are finite abelian we have that \[ \gamma(Pic(O)[2^{\infty}]) = C\ell(K)[2^{\infty}] \ {\rm and} \ \beta(Pic^{+}(O)[2^{\infty}]) = C\ell^{+}(K)[2^{\infty}].\] The first claim follows from the above equalities and from $\ker(\beta) \cong \ker(\gamma)$. Since $\displaystyle {\rm rk}_{4}(Pic^{+}(O))={\rm rk}_{4}(C\ell^{+}(K))$ whenever $\displaystyle Pic^{+}(O)[2^{\infty}] \cong C\ell^{+}(K)[2^{\infty}]$ the second claim follows from the first.
\end{proof}

\begin{lemma}\label{3mod4disc}
Let $d$ be a square free integer and let $D=-4d$. If $d \not \equiv 3 \pmod{4}$ then $D$ is a fundamental discriminant, and every fundamental discriminant $D \equiv 0 \bmod 4$ is of this form. If $d  \equiv 3 \pmod{4}$ then $-d$ fundamental and ${\rm rk}_{4}(C_{D})={\rm rk}_{4}(C_{-d})$. 
\end{lemma}

\begin{proof}
The first assertion is clear so we may assume that $d \equiv 3 \pmod{4}$ so that $C_{-d}$ is the narrow class group of $K=\Q(\sqrt{-d})$, and $-d={\rm Disc}(K)$ is the discriminant of $K$. Since $D=2^2{\rm Disc}(K)$, $D$ is the discriminant of the order of conductor $2$ in $K.$ Thanks to Proposition \ref{D2} it is enough to show that \[Cl_{D}[2^{\infty}] \cong Cl_{-d}[2^{\infty}]\] (where $Cl_{\delta}$ denotes the ideal class group of the order of discriminant $\delta$). From \cite[I, \S12 Theorem 12]{neu} we see that $\#Cl_{D}=3^{\epsilon}\#Cl_{-d}$ where $\epsilon \in \{0,1\}$. Thus $Cl_{D}[2^{\infty}] \cong Cl_{-d}[2^{\infty}]$. \end{proof}

For $i=1,2,3$ we define the quantities $T^{\pm}_{i}(X)$ by restricting the count $T^{\pm}(X)$ to only square free integers $d \equiv i \pmod{4}.$ 

\begin{proposition}\label{asymp} Suppose that in the following $D$ is reserved for fundamental discriminants. Then,
\begin{align*}
T^{+}_{1}(X) &= \left | \{0< -D < 4X,\  D \equiv 12 \bmod{16}, \ {\rm rk}_{4}(C_{D})=0 \} \right | \\
T^{-}_{1}(X) &= \left | \{0< \ D \ < 4X,\  D \equiv 12 \bmod{16}, \ {\rm rk}_{4}(C_{D})=0 \} \right | \\
T^{+}_{2}(X) &= \left | \{0< -D < 4X,\  D \equiv 8 \bmod{16}, \ {\rm rk}_{4}(C_{D})=0 \} \right | \\
T^{-}_{2}(X) &=  \left | \{0< \ D \ < 4X,\  D \equiv 8 \bmod{16}, \ {\rm rk}_{4}(C_{D})=0 \} \right | \\
T^{+}_{3}(X) &= \left | \{0< -D < X,\  D \equiv 1 \bmod{4}, \ {\rm rk}_{4}(C_{D})=0 \} \right | \\
T^{-}_{3}(X) &= \left | \{0< \ D \ < X,\  D \equiv 1 \bmod{4}, \ {\rm rk}_{4}(C_{D})=0 \} \right |.
\end{align*}Furthermore,
\begin{align*}
T^{+}_{1}(X)&=\frac{2\phi(1/2)}{\pi^2}X +o(X), T^{-}_{1}(X)=\frac{4\phi(1/2)}{\pi^2}X +o(X)\\ T^{+}_{2}(X)&=\frac{2\phi(1/2)}{\pi^2}X +o(X), T^{-}_{2}(X)=\frac{4\phi(1/2)}{\pi^2}X +o(X)\\ T^{+}_{3}(X)&=\frac{2\phi(1/2)}{\pi^2}X +o(X), T^{-}_{3}(X)=\frac{4\phi(1/2)}{\pi^2}X +o(X),  
\end{align*} and
 \[T^{+}(X)=\frac{6\phi(1/2)}{\pi^2}X+o(X), \ T^{-}(X)=\frac{12\phi(1/2)}{\pi^2}X+o(X). \]
\end{proposition}

\begin{proof}
The first part of the proposition follows from Proposition \ref{4ranktheorem} and Lemma \ref{3mod4disc}, and the second is obtained thanks to Theorem \ref{FK}. Since $T^{\pm}(X)=T^{\pm}_{1}(X)+T^{\pm}_{2}(X)+T^{\pm}_{3}(X)$ the last part of the proposition follows.
\end{proof}
Now we are ready to prove Theorem \ref{teoremacaso2}
\begin{proof}
Since $N^{+}(X)=\frac{6}{\pi^2}X+o(X), N^{-}(X)=\frac{6}{\pi^2}X+o(X)$ and $N(X)=\frac{12}{\pi^2}X+o(X)$ Propositions \ref{asymp} implies that
\begin{align*}
\frac{T^{+}(X)}{N^{+}(X)} &= \frac{\frac{6\phi(1/2)}{\pi^2}X+o(X)}{\frac{6}{\pi^2}X+o(X)}\\
\frac{T^{-}(X)}{N^{-}(X)} &= \frac{\frac{12\phi(1/2)}{\pi^2}X+o(X)}{\frac{6}{\pi^2}X+o(X)} \\
\frac{T(X)}{N(X)} &= \frac{\frac{18\phi(1/2)}{\pi^2}X+o(X)}{\frac{12}{\pi^2}X+o(X)},
\end{align*} from which the theorem follows.
\end{proof}

\section*{Acknowledgements}

In the first place I would like to thank the referee for the careful reading of the paper, and specially for point it out the existence of Theorem \ref{estespall}. I would like to thank Eva Bayer-Fluckiger for many helpful conversations around the topic of this paper, and for her thorough and valuable comments on an earlier draft of it. I also thank to Lisa (Powers) Larsson for her helpful comments on one of the first drafts of the paper.

\noindent
Guillermo Mantilla-Soler\\
Departamento de Matem\'aticas,\\
Universidad de los Andes, \\
Carrera 1 N. 18A - 10, Bogot\'a, \\
Colombia.\\
g.mantilla691@uniandes.edu.co

\end{document}